\documentclass[a4paper,10pt]{article}
\usepackage{graphics}
\usepackage{amssymb}
\usepackage{amsmath}
\usepackage{amsthm}
\usepackage{latexsym}
\theoremstyle{plain}
\newtheorem{proposition}{Proposition}
\newtheorem{lem}[proposition]{Lemma}
\newtheorem*{corollary}{Corollary}
\newtheorem{theorem}[proposition]{Theorem}
\theoremstyle{definition}
\newtheorem{example}[proposition]{Example}
\newtheorem{remark}[proposition]{Remark}
\newtheorem*{acknowledgement}{Acknowledgement}
\newtheorem{defn}[proposition]{Definition}

\begin{document}
\def\thefootnote{}
\footnotetext{\!\!\!2000 Mathematics Subject Classification: 14H50, 14E07}
\def\thefootnote{1}

\title{Hyperelliptic plane curves of type $(d,d-2)$}
\author{Fumio SAKAI, Mohammad SALEEM\footnote{%
Partially supported by Post Doctoral Fellowship for Foreign Researchers, 1705292, JSPS.}
 \,\,and Keita TONO}
\date{}
\maketitle
\begin{abstract}
In \cite{SS}, we classified and constructed all rational plane curves of type $(d,d-2)$. In this paper, we generalize these results to irreducible plane curves of type $(d,d-2)$ with positive genus.
\end{abstract}

\section{Introduction}

Let $C\subset {\mathbf{P}}^2={\mathbf{P}}^2(\mathbf{C})$
be a plane curve of degree $d$.
We call $C$ a {\sl plane curve of type} $(d,\nu)$
if the maximal multiplicity of singular points on $C$
is equal to $\nu$.
A unibranched singularity is called a \textsl{cusp}.
Rational cuspidal plane curves of type $(d,d-2)$ and $(d,d-3)$
were classified by Flenner--Zaidenberg \cite{FZ1,FZ2}
(See also \cite{F, ST} for some cases).
In \cite{SS}, we classified rational plane curves of type $(d,d-2)$
with arbitrary singularities.
In order to describe a multibranched singularity $P$,
we introduced the notion of the system of the multiplicity sequences
$\underline{m}_P(C)$ (See Sect.2).
We denote by Data$(C)$, the collection of such systems
of multiplicity sequences.
The purpose of this paper is to complete the classification
of irreducible plane curves of type $(d,d-2)$
with positive genus $g$.
We remark that if $g\geq 2$, then $C$ is a hyperelliptic curve,
for the projection from the singular point of multiplicity $d-2$
induces a double covering of $C$ over $\mathbf{P}^1$.     

\begin{theorem}\label{theorem 1}Let $C$ be a plane curve of
type $(d,d-2)$ with genus $g$.
Let $Q\in C$ be the singular point of multiplicity $d-2$.
Then, we have 
\begin{enumerate}
\item[{\rm (i)}]
{\rm Data}$(C) =\left[ \underline{m}_{Q}( C) ,{\binom{1}{1}}_{b_{1}},\ldots, {\binom{1}{1}}_{b_{n}},(2_{b_{n+1}}), \ldots, (2_{b_{n+n'}})\right]$,
where 
\[
\underline{m}_{Q}(C) =\left\{
\!\begin{array}{ll}
\left(
\def\arraystretch{1.25}
\!\!\begin{array}{l}
k_1\\
k'_1\\
\,\vdots\\
k_s\\
k'_s\\
k_{s+1}\\
\,\vdots\\
k_N
\end{array}\!\!\!\!\right)
&
\def\arraystretch{1.127}
\!\!\!\!\begin{array}{c}
\!\!\!\!\left(\!\!\!
{\scriptscriptstyle
\begin{array}{l}
1\\
1
\end{array}}\!\!\!\right)_{a_1}\\
\!\!\!\!\vdots\quad{}\\
\!\!\!\!\left(\!\!\!
{\scriptscriptstyle
\begin{array}{l}
1\\
1
\end{array}}\!\!\!\right)_{a_s}\\
\,2_{a_{s+1}}\\
\!\!\!\!\raisebox{2pt}{$\vdots$}\quad{}\\
\!\!\!2_{a_{\raisebox{-.5pt}{$\scriptscriptstyle N$}}}
\end{array}
\end{array}
\def\arraystretch{1}
\!\!\!\!\right\}\qquad\qquad\qquad\qquad
\]
and the following conditions are satisfied: 
\begin{enumerate}
\item[{\rm (1)}]
${\displaystyle \sum_{h=1}^Nk_{h}+\sum_{h^{\prime }=1}^sk'_{h'}}=d-2$
and
${\displaystyle \sum_{i=1}^{N} a_{i}+\sum_{j=1}^{n+n'}b_j}=d-g-2$,
where  $a_{i}\geq 0$ \,\,($a_i>0$ for $i=1,\ldots,s$), $b_{j}>0$,
\item[{\rm (2)}]
we have $n, n', s \geq 0$ and $n'+s'\leq 2g+2$,
where $s'=\#\{j|a_{s+j}>0\}$,
\item[{\rm (3)}]
for $i=1,2,\ldots,s$, if $k'_{i}=k_{i}$,
then $a_{i}\geq k_{i}$ and if $k'_{i}>k_{i}$, then $a_{i}=k_{i}$, 
\item[{\rm (4)}]
for $i=s+1, \ldots, N$, if $a_{i}> 0$, then either $k_{i}$
is even and $a_{i}\geq k_{i}/2$ or $k_{i}$ is odd and $a_{i}=(k_{i}-1)/2$. 
\end{enumerate}
Note that the $N$ is the number of the different tangent lines to $C$ at $Q$. 
\item[{\rm (ii)}]
Data$(C)$ can be derived from Degtyarev's 2--formula $T(C)$
defined for the defining equation of $C$
(See Proposition~\ref{formula} for details). 
\end{enumerate}
\end{theorem}
\begin{corollary}\label{cor:2} Let $C$ be an irreducible plane curve of type 
$(d,d-2)$ with genus $g$. 
\begin{enumerate} \item[{\rm (i)}]  If $C$ has only cusps, then $C$ has the following data ($b_{i}> 0, k>0, j\geq 0$):
\[
\renewcommand{\arraystretch}{1.5}%
\begin{array}{c|ll}\hline
\mathrm{Class} & \mathrm{Data}(C) &  \\ \hline
\mathrm{(a)}
& [(k), (2_{b_{1}}), \ldots, (2_{b_{n'}})]
&(k=g+\sum_{i=1}^{n'}b_{i})\\
&&(n'\leq 2g+2)\\
\mathrm{(b)}
& [(2k+1, 2_{k}), (2_{b_{1}}), \ldots ,(2_{b_{n'}})]& (k+1=g+\sum_{i=1}^{n'}b_{i})\\
&&(n'\leq 2g+1)\\
\mathrm{(c)} & [(2k, 2_{k+j}), (2_{b_{1}}), \ldots, (2_{b_{n'}})]& (k=g+j+\sum_{i=1}^{n'}b_{i})\\
&&(n'\leq 2g+1)\\\hline
\end{array}
\]
\item[{\rm (ii)}]
If $C$ has only bibranched singularities,
then $C$ has the following data ($b_{i}>0, k>0, r>0, j\geq 0, l\geq 0$):
\[
\def\arraystretch{1.6}
\begin{array}{c|ll}\hline
\mathrm{Class} & \mathrm{Data}(C) &  \\ \hline
\mathrm{(e)}
&
\Bigl[\binom{k}{k}\binom{1}{1}_{k+j}, {\binom{1}{1}}_{b_{1}},\ldots, {\binom{1}{1}}_{b_{n}}\Bigr]
& (k=g+j+\sum_{i=1}^{n}b_{i}) \\
\mathrm{(f)}
&
\Bigl[\binom{k}{k+r}\binom{1}{1}_{k}, {\binom{1}{1}}_{b_{1}},\ldots, {\binom{1}{1}}_{b_{n}}\Bigr]
& (k+r=g+\sum_{i=1}^{n}b_{i}) \\
\mathrm{(aa)}
&
\Bigl[\binom{k}{r}, {\binom{1}{1}}_{b_{1}},\ldots, {\binom{1}{1}}_{b_{n}}\Bigr]
& (k+r=g+\sum_{i=1}^{n}b_{i}) \\
\mathrm{(ab)}
&
\Bigl[{\bigl\{
\binom{2k+1}{r}
\def\arraystretch{1}
\!\!\begin{array}{l}
\raisebox{-2pt}{$\scriptstyle 2_{k}$}\\
{}
\end{array}\!\!\bigr\}}, {\binom{1}{1}}_{b_{1}},\ldots, {\binom{1}{1}}_{b_{n}}\Bigr]
& (k+r+1=g+\sum_{i=1}^{n}b_{i}) \\
\mathrm{(ac)}
&
\Bigl[{\bigl\{
\binom{2k}{r}
\def\arraystretch{1}
\!\!\begin{array}{l}
\raisebox{-2pt}{$\scriptstyle 2_{k+j}$}\\
{}
\end{array}\!\!\bigr\}}, {\binom{1}{1}}_{b_{1}},\ldots, {\binom{1}{1}}_{b_{n}}\Bigr]
& (k+r=g+j+\sum_{i=1}^{n}b_{i}) \\
\mathrm{(bb)}
&
\Bigl[{\bigl\{
\binom{2k+1}{2r+1}
\def\arraystretch{1}
\!\!\begin{array}{l}
\raisebox{-2pt}{$\scriptstyle 2_{k}$}\\
\raisebox{2.5pt}{$\scriptstyle 2_{r}$}
\end{array}\!\!\bigr\}}, {\binom{1}{1}}_{b_{1}},\ldots, {\binom{1}{1}}_{b_{n}}\Bigr]
& (k+r+2=g+\sum_{i=1}^{n}b_{i}) \\
\mathrm{(bc)}
&
\Bigl[{\bigl\{
\binom{2k+1}{2r}
\def\arraystretch{1}
\!\!\begin{array}{l}
\raisebox{-1.8pt}{$\scriptstyle 2_{k}$}\\
\raisebox{2.5pt}{$\scriptstyle 2_{r+l}$}
\end{array}\!\!\bigr\}}, {\binom{1}{1}}_{b_{1}},\ldots, {\binom{1}{1}}_{b_{n}}\Bigr]
& (k+r+1=g+l+\sum_{i=1}^{n}b_{i}) \\
\mathrm{(cc)}
&
\Bigl[{\bigl\{
\binom{2k}{2r}
\def\arraystretch{1}
\!\!\begin{array}{l}
\raisebox{-2.2pt}{$\scriptstyle 2_{k+j}$}\\
\raisebox{2.8pt}{$\scriptstyle 2_{r+l}$}
\end{array}\!\!\bigr\}}, {\binom{1}{1}}_{b_{1}},\ldots, {\binom{1}{1}}_{b_{n}}\Bigr]
& (k+r=g+j+l+\sum_{i=1}^{n}b_{i}) \\\hline

\end{array}
\]
\end{enumerate}
\end{corollary}

\begin{theorem}[Cf. Coble \cite{Cob}, Coolidge \cite{Co}]\label{theorem 2}
Let $C$ be an irreducible plane curve of
type $(d,d-2)$ with genus $g$.
Then, there exists a Cremona transformation
which transforms $C$ into a plane curve:
\[
\Gamma: y^2=\prod_{i=1}^{2g+2}(x-\lambda_i),
\]
with some distinct $\lambda_i$'s. 

Conversely, given a plane curve $\Gamma$ as above
and a collection of systems of multiplicity sequences $M$
satisfying the conditions (1)--(4) in Theorem~\ref{theorem 1}, (i)
for $d\geq g+2$,
then we can find an irreducible plane curve $C$ of type $(d,d-2)$ such that
\begin{enumerate} \item[\rm (a)] Data$(C)=M$,
\item[\rm (b)] $C$ is Cremona birational to $\Gamma$.
\end{enumerate}
\end{theorem}

In Sect.2, we review the system of the multiplicity sequences,
the 2--formula and quadratic Cremona transformations.
In Sect.3 (resp. Sect.4), we will prove Theorem~\ref{theorem 1}
(resp. Theorem~\ref{theorem 2}).
In Sect.5, we discuss the defining equations for those curves given in
Corollary.
\section{Preliminaries}\label{Prel} 
A cusp $P$ can be described by its multiplicity sequence
$\underline{m}_P=(m_0,m_1,m_2,\ldots)$.
For a multibranched singular point $P$ on $C$,
we introduced {\sl the system of the multiplicity sequences} of $P$. 

\begin{defn}[\cite{SS}]
Let $P\in C$ be a multibranched singular point,
having $r$ local branches $\gamma_1,\ldots,\gamma_r$.
Let $\underline{m}(\gamma_i)=(m_{i0}, m_{i1}, m_{i2}, \ldots)$ denote
the multiplicity sequences of the branches $\gamma_i$, respectively.
We define the \textsl{system of the multiplicity sequences},
which will be denoted by the same symbol $\underline{m}_P(C)$,
to be the combination of $\underline{m}(\gamma_i)$
with brackets indicating the coincidence of the centers of
the infinitely near points of the branches $\gamma_i$.
For instance, for the case in which $r=3$, we write it in the following form: 
\[
{
\!\left\{\!\left(\!\!\!
\def\arraystretch{1.2}
\begin{array}{l}
m_{1,0}\\
m_{2,0}\\
m_{3,0}
\end{array}
\!\!\!\right)
\ldots
\left(\!\!\!
\def\arraystretch{1.2}
\begin{array}{l}
m_{1,\rho }\\
m_{2,\rho }\\
m_{3,\rho }
\end{array}\!\!\!\right)\!\setlength{\tabcolsep}{3pt}\begin{array}{l}
\!\!\left(\!\!\!\!
\def\arraystretch{1.2}
\begin{array}{l}
\raisebox{1pt}{$m_{1,\rho+1}$}\\
\raisebox{1pt}{$m_{2,\rho+1}$}
\end{array}\!\!\!\!\right)\ldots
\left(\!\!\!\!
\def\arraystretch{1.2}
\begin{array}{l}
\raisebox{1pt}{$m_{1,\rho'}$}\\
\raisebox{1pt}{$m_{2,\rho'}$}
\end{array}\!\!\!\!\!\right)\
\!\!\!\!\renewcommand{\arraystretch}{1}\setlength{\tabcolsep}{3pt}\begin{array}{l}
\raisebox{1.6pt}{$m_{1,\rho'+1},\!\ldots\!,\!m_{1,s_1}$}\\
\raisebox{-0.6pt}{$m_{2,\rho'+1},\!\ldots\!,\!m_{2,s_2}$}
\end{array}\\
\,\, m_{3,\rho+1},\ldots,  m_{3,s_{3}}
\end{array}
\!\!\!\!\!\!\right\}}.
\def\arraystretch{1}
\]
We also use some simplifications such as 
\[
(2_a)=(\overbrace{2,\ldots,2}^a,1,1), \quad (2_0)=(1), \quad \binom{1}{1}_a=\overbrace{\binom{1}{1} \ldots \binom{1}{1}}^a.
\] 
\end{defn}

\begin{example}
We examine our notations for ADE singularities.
\[
\def\arraystretch{2}
\begin{array}{c|c|c|c|c|c|c|c}\
P&A_{2n}&A_{2n-1}&D_{2n-1}&D_{2n}&E_6&E_7&E_8\\\hline
\underline{m}_P(C)&(2_n)&{\binom{1}{1}_n}&
{\bigl\{
\binom{2}{1}
\def\arraystretch{1}
\!\!\begin{array}{l}
\raisebox{-1.8pt}{$\scriptstyle 2_{n-3}$}\\
{}
\end{array}\!\!\bigr\}}
&
\raisebox{-2pt}{$
\left\{\!\left(
\def\arraystretch{.8}
\!\!\!\begin{array}{c}
\scriptstyle 1\\
\scriptstyle 1\\
\scriptstyle 1
\end{array}\!\!\!\right)
\def\arraystretch{.8}
\!\!\!\!\begin{array}{l}
\left(\!\!\!\begin{array}{l}
\scriptstyle 1\\
\scriptstyle 1
\end{array}\!\!\!\right)_{\scriptscriptstyle n-2}\\
{}
\end{array}
\!\!\!\right\}$}
&(3)&\binom{2}{1}\binom{1}{1}&
(3,2)\end{array}
\def\arraystretch{1}
\]
\end{example}

\begin{example} The hyperelliptic curve $y^2=\prod_{i=1}^{2g+2}(x-\lambda_i)$ has one singularity $Q$ on the line at infinity with $\underline{m}_Q=\binom{g}{g}\binom{1}{1}_g$.
\end{example}

Let $C$ be an irreducible plane curve of type $(d,d-2)$. Let $Q\in C$ be the singular point with multiplicity $d-2$. Choosing homogeneous coordinates $(x,y,z)$ so that $Q=(0,0,1)$, the curve $C$ is defined by an equation:
\[
F(x,y)z^2+2G(x,y)z+H(x,y)=0,
\]
where $F$, $G$ and $H$ are homogeneous polynomials of degree $d-2$,  $d-1$ and $d$, respectively.  Set $\Delta=G^2-FH$. Let $t_1, \ldots, t_l\in{\mathbf{P}}^1$ be all the distinct roots of
the equation  $F(t)\Delta(t)=0$.
For each $i$, let $(p_i,q_i)=(\mathrm{ord}_{t_i}(F),\mathrm{ord}_{t_i}(\Delta))$,
where $\mathrm{ord}_{t_i}(F)$ (resp. $\mathrm{ord}_{t_i}(\Delta)$) is the multiplicity of the root $t_i$
of the equation $F(t)=0$ (resp. $\Delta(t_i)=0$). Set $T(C)=\{(p_1,q_1),\ldots,(p_l,q_l)\}$.  This unordered $l$--tuple $T(C)$ is called the \textsl{2--formula} of $C$ (Degtyarev \cite{D}). 
We remark that $T(C)$ does not depend on the choice of the coordinates $(x,y,z)$ with $Q=(0,0,1)$. 
\begin{lem}\label{2FL} The 2--formula $T(C)$ satisfies the following properties:
\begin{enumerate} 
\item[\rm{(i)}] $\displaystyle\sum_{i=1}^l q_i = 2\sum_{i=1}^l p_i+2$, \label{Def_ITF1}
\item[\rm{(ii)}] $p_i=q_i$ or $\min\{p_i,q_i\}$ is even for each $i$, 
\item[\rm{(iii)}] there exists a pair $(p_i,q_i)$ such that $q_i$ is an odd number.\label{Def_ITF3}
\end{enumerate}
\end{lem}
\begin{proof} (i) By definition, $\sum_{i=1}^l p_i=d-2$ and $\sum_{i=1}^l q_i=2d-2$. (ii) Suppose that $p_i\ne q_i$. We may assume $t_i=(0,1)$.
We can write $F$, $G$ and $\Delta$ as $F=x^{p_i}F_0$, $\Delta=x^{q_i}\Delta_0$ and $G=x^mG_0$, respectively,
where $m=\mathrm{ord}_{t_i}(G)$. If $p_i>q_i$, then $x^{2m}G_0^2=x^{q_i}(\Delta_0+x^{p_i-q_i}F_0H)$.
Thus $q_i=2m$. If $0<p_i<q_i$, then we get $x^{2m}G_0^2=x^{p_i}(x^{q_i-p_i}\Delta_0+F_0H)$, which implies that $2m\ge p_i>0$. We have $x\,{\not\kern0pt|}\,\,H$, since $C$ is irreducible.
Hence $p_i=2m$. (iii) Suppose that all $q_i$'s are even. Then we can write as $\Delta=\Delta_0^2$. We have $F(Fz^2+2Gz+H)=(Fz+G+\Delta_0)(Fz+G-\Delta_0)$.
Since $\deg(Fz+G\pm\Delta_0)=d-1$, we infer that 
$Fz^2+2Gz+H$ is reducible. This is a contradiction.
\end{proof}
\begin{remark}\label{criterion} We note that $P(x,y,z)=Fz^2+2Gz+H$ is irreducible if (a) $\mathrm{GCD}(F,G,H)=1$, and if (b) the property (iii) holds. Indeed, under the assumption (a), if $P$ is reducible, then $P=(Az+B)(Cz+D)$ with $A,B,C,D\in {\mathbf{C}}[x,y]$. But, in this case, $4\Delta=(AD-BC)^2$, which contradicts the property~(iii).  
\end{remark}
\begin{example}
Let $C$ be the quartic curve $x^2y^2+y^2z^2+z^2x^2-2xyz(x+y+z)=0$.
We have $T(C)=\{(2,0),(0,3),(0,3)\}$.
\end{example}

The (degenerate) quadratic Cremona transformation 
\[
\varphi _{c}:(x,y,z)\longrightarrow (xy,y^{2},x(z-cx)) \quad (c \in \mathbf{C})
\]
played an important role in \cite{FZ2,SS}. We find
that $\varphi _{c}^{-1}(x,y,z)=(x^{2},xy,yz+cx^{2})$. We use the notations: 
\[
l:x=0, \,\,t:y=0, \,\,O=(0,0,1), A=(1,0,c), B=(0,1,0).
\]
Note that $\varphi_c(l\setminus \{O\})=B$ and $\varphi_c(t\setminus \{O,A\})=O$. 
 
Let $C$ be an irreducible plane curve of type $(d,d-2)$ with $d\geq 4$.
Suppose the singular point $Q\in C$ of multiplicity $d-2$ has coordinates $O$.
We have seen  in \cite{SS} that the strict transform $C'=\varphi_c(C)$
is an irreducible plane curve of type $(d',d'-2)$ for some $d'$.
In \cite{SS,ST}, by analyzing how a local branch $\gamma$ at
$P\in \textrm{Sing}(C)$ is transformed by $\varphi_c$,
we described Data$[C']$ from Data$[C]$.     

\section{Proof of Theorem 1}
(i)
We easily see that $P\in \textrm{Sing}(C)\setminus \{Q\}$ is a double point,
because $LC=(d-2)Q+2P$, where $L$ is the line passing through $P,Q$.
Let $\pi:{\tilde {\mathbf{P}}}^2\to {\mathbf{P}}^2$ be the blowing--up at $Q$.
Let $E$ denote the exceptional curve.
Take a line $L$ passing through $Q$.
Let $C'$ (resp. $L'$) be the strict transform of $C$ (resp. $L$).
We have $C'L'=2$.
It follows that $P\in \textrm{Sing}(C')\cap E$ is also a double point.
Thus, Data$(C)$ has the shape as in Theorem~\ref{theorem 1}.
Clearly, $\sum k_h+\sum k'_{h'}=\textrm{mult}_Q(C)=d-2$.
The second part of the condition (1) follows from the genus formula.
The condition~(2) follows from the Hurwitz formula applied to
the double covering ${\tilde C}\to {\mathbf{P}}^1$,
which corresponds to the projection of $C$ from $Q$,
where the ${\tilde C}$ is the non--singular model of $C$.
For the proof of the conditions~(3), (4), we refer to \cite{SS}.
We will give an alternative, direct proof in Proposition~\ref{formula}. 

\bigskip
\noindent
(ii) Let $F(x,y)z^2+2G(x,y)z+H(x,y)=0$ be the defining equation of $C$ 
as in Sect.~2. Let $T(C)$ be the 2--formula of $C$. Setting 
\[
T'(C)=\{(p,q)\in T(C)\,|\, p>0 \textrm{ or } q\geq 2\},
\]
we renumber the pairs $(p_i,q_i)\in T'(C)$ in the following way:
\begin{enumerate}
\item[\rm{(1)}] $p_i>0$, $q_i>0$ and $q_i$ is even for $i=1,\ldots s$, 
\item[\rm{(2)}] either $p_i>0$, $q_i>0$ and $q_i$ is odd, or $p_i>0, q_i=0$ for $i=s+1, \ldots, N$,
\item[\rm{(3)}] $p_i=0$, $q_i>0$ and $q_i$ is even, for $i=N+1,\ldots,N+n$,
\item[\rm{(4)}] $p_i=0$, $q_i\geq 3$ and $q_i$ is odd, for $i=N+n+1,\ldots,N+n+n'$.
\end{enumerate}

\bigskip
\begin{proposition}\label{formula} Set 
\begin{enumerate}
\item[\rm{(1)}] for $i=1,\ldots s$, 

\noindent
$\begin{cases}
k_i=k'_i=p_i/2, a_i=q_i/2 &\textrm{ if } p_i\leq q_i,\\
k_i=q_i/2, k'_i=p_i-q_i/2, a_i=q_i/2 &\textrm{ if } p_i>q_i,
\end{cases}$

\item[\rm{(2)}] for $i=s+1,\ldots N$,

\noindent
$\begin{cases}
k_i=p_i, a_i=(q_i-1)/2 &\textrm{ if } q_i>0,\\
k_i=p_i, a_i=0 &\textrm{ if } q_i=0,
\end{cases}$
\item[\rm{(3)}] $b_j=q_{\raisebox{-2.5pt}{{\footnotesize N+j}}}/2$, for $j=1,\ldots,n$, 
\item[\rm{(4)}] $b_j=(q_{\raisebox{-2.5pt}{{\footnotesize N+j}}}-1)/2$. for $j=n+1,\ldots,n+n'$. 
\end{enumerate}
Then Data$(C)$ is given as in Theorem~\ref{theorem 1}, (i).
\end{proposition}

\begin{proof} Take $(p_i,q_i)\in T'(C)$. Write $t_i$ as $t_i=(\alpha_i,\beta_i)$.
Let $L_i$ be the line $\beta_ix=\alpha_iy$. By arranging the coordinates, we may assume $(\alpha_i,\beta_i)=(0,1)$.
Write $F$, $G$ and $\Delta$ as $F=x^{p_i}F_0$, $G=x^mG_0$ and $\Delta=x^{q_i}\Delta_0$, where $m=\mathrm{ord}_{t_i}(G)$. 

We first consider the case in which $p_i=0$. Since $\Delta(t_i)=0$, we have 
\[
F(t_i)z^2+2G(t_i)z+H(t_i)=F(t_i)(z+G(t_i)/F(t_i))^2.
\]
It follows that $CL_i=(d-2)Q+2P$, where $P=(0,1,-G(t_i)/F(t_i))$.
Let $U$ be a neighbourhood of $P$ such that $y\ne 0$ and $F(x,y)\ne0$ for all $(x,y,z)\in U$. We use the affine coordinates $(\overline{x},\overline{z})=(x/y,z/y)$. We have
\[
\begin{split}
&F(x,y)(F(x,y)z^2+2G(x,y)z+H(x,y))\\&\qquad \qquad \qquad =y^{2d-2}((F(\overline{x},1)\overline{z}+G(\overline{x},1))^2-\Delta(\overline{x},1)).
\end{split}
\]
Thus $C$ is defined by the equation
$(F(\overline{x},1)\overline{z}+G(\overline{x},1))^2=\Delta(\overline{x},1)$
on $U$. Letting $u=F(\overline{x},1)\overline{z}+G(\overline{x},1)$
and $v=(\sqrt[q_i]{\Delta_0(\overline{x},1)})\overline{x}$,
$C$ is defined by the equation $u^2=v^{q_i}$ around $P$.
Thus $P\in{\mathop{\rm Sing}\nolimits} (C)\setminus\{Q\}$ if $q_i\ge 2$.  In this case, we have 
\[
\underline{m}_P(C)=\begin{cases} \binom{1}{1}_{q_i/2} &\textrm{ if } q_i  \textrm{ is even}, \\
(2_{(q_i-1)/2}) & \textrm{ if } q_i \textrm{ is odd},
\end{cases}
\]
which gives the assertions (3), (4). 

Conversely, take $P\in{\mathop{\rm Sing}\nolimits} (C)\setminus\{Q\}$. Let $L$ be the line passing through $P, Q$. Write $L: \beta x=\alpha y$.
Since $CL=(d-2)Q+2P$, we have $F(\alpha,\beta)\ne0$
and $\Delta(\alpha,\beta)=0$. For $(\alpha,\beta)\in {\mathbf{P}}^1$,
we find a pair $(0,q)\in T(C)$.
We see from the above argument that $C$ is defined by the equation $u^2=v^q$ near $P$.
Thus $q\ge 2$.

We now consider the case in which $p_i>0$.
Let $\pi:{\tilde {\mathbf{P}}}^2\rightarrow {\mathbf{P}}^2$ be
the blowing--up at $Q$ and $E$ the exceptional curve of $\pi$.
We use the affine coordinates $(\overline{x},\overline{y})=(x/z,y/z)$ of $U:=\{(x,y,z)\in{\mathbf{P}}^2\,\,|\,\,z\ne0\}$. 
Put $V=\pi^{-1}(U)$. There exist an open cover $V=V_1\cup V_2$ ($V_j\cong{\mathbf{C}}^2$) with standard coordinates $(u_j,v_j)$ of $V_j$ such that
$\pi|_{V_1}:V_1\ni(u_1,v_1)\mapsto(u_1v_1,u_1)$ and
$\pi|_{V_2}:V_2\ni(u_2,v_2)\mapsto(u_2,u_2v_2)$. Note that $E$ is defined by $u_j=0$ on $V_j$.
The strict transform $L_i'$ of $L_i$
is defined by $v_1=0$ on $V_1$. Let $P$ be the unique point $E\cap L_i'$.
We have $P=(0,0)$ on $V_1$. The strict transform $C'$ of $C$ is defined by the equation:
$F(v_1,1)+2G(v_1,1)u_1+H(v_1,1)u_1^2=0$ on $V_1$.
By the definition of $p_i$ and $m$, the curve $C'$ is defined by the equation:
$F_0v_1^{p_i}+2G_0v_1^mu_1+Hu_1^2=0$.
In particular, we have $(C'E)_P=p_i$.
If $q_i=0$, then we must have $m=0$ (See the proof of Lemma~\ref{2FL}).
Hence $C'$ is smooth at $P$.
If $q_i>0$, then we have $m>0$ (Cf. the proof of Lemma~\ref{2FL}).
Since $C$ is irreducible, we see that $H(t_i)\ne0$.
We have $H(F+2G_0v_1^mu_1+Hu_1^2)=(Hu_1+G_0v_1^m)^2-\Delta$.
This means that $C'$ is defined by the equation:
\[
(H(v_1,1)u_1+G_0(v_1,1)v_1^m)^2-\Delta(v_1,1)=0
\]
in a neighborhood of $P$.
Letting $u=H(v_1,1)u_1+G_0(v_1,1)v_1^m$ and $v=(\sqrt[q_i]{\Delta_0(v_1,1)})v_1$, $C'$ is defined by the equation $u^2=v^{q_i}$ around $P$.
We have
\[
\underline{m}_P(C')=\begin{cases} \binom{1}{1}_{q_i/2} &\textrm{ if } q_i  \textrm{ is even}, \\
(2_{(q_i-1)/2}) & \textrm{ if } q_i \textrm{ is odd},\\
(1) &\textrm{ if } q_i=0,
\end{cases}
\]
which gives the values of $a_i$ in (1), (2).
We prove the remaining assertions in (1).
If $q_i$ is even, then $C'$ has two branches $\gamma_{+}, \gamma_{-}$ at $P$ defined by 
\[
H(v_1,1)u_1+G_0(v_1,1)v_1^m\pm v_1^{q_i/2}\sqrt{\Delta_0(v_1,1)}=0.
\]
In case $p_i>q_i$, we have $m=q_i/2$ (See the proof of Lemma~\ref{2FL}).
We infer that one of the intersection numbers
$(E\gamma_{+})_P$ and $(E\gamma_{-})_P$ is equal to $q_i/2$.
The other one must be equal to $p_i-q_i/2$, because $(EC')_P=p_i$.
In case $p_i\leq q_i$, we have $m\geq p_i/2$
(Cf. the proof of Lemma~\ref{2FL}).
Thus $(E\gamma_{\pm})_P\geq p_i/2$, hence $(E\gamma_{\pm})_P=p_i/2$.
Consequently, we obtain the pair $(k_i,k'_i)$.

Conversely, take $P\in C'\cap E$. We assume $P\in V_1$.
Write the coordinates of $P$ as $P=(0,\beta)$.
The equation $F(\beta,1)+2G(\beta,1)u_1+H(\beta,1)u_1^2=0$
has the solution $u_1=0$ as $C'$ passes through $P$.
Thus $F(\beta,1)=0$. For $(\beta,1)\in {\mathbf{P}}^1$, we find a pair $(p,q)\in T(C)$ with $p>0$. 
\end{proof}

\begin{remark}
For $i=s+1,\ldots,N$, if $(k_i,a_i)=(1,0)$,
then we have either $(p_i,q_i)=(1,1)$ or $(1,0)$.
The case $(p_i,q_i)=(1,0)$ occurs if and only if the line $L_i$ is a flex--tangent line to the corresponding branch at $Q$. 
\end{remark}

\section{Proof of Theorem 2}
Let $C$ be given by the equation (See Sect.~\ref{Prel}):
\[
F(x,y)z^2+2G(x,y)z+H(x,y)=0.
\]
Put $\Delta=G^2-FH$.
Via linear coordinates change of $x$ and $y$,
we may assume that $y\,{\not\kern0pt|}\,\,F\Delta$.
We then define a Cremona transformation (Cf. \cite{Co}, Book II, Chap.V):
\[
\Phi(x,y,z)=(xy^{d-2},y^{d-1}, Fz+G),
\]
We find that $\Phi^{-1}(x,y,z)=(xF,yF, y^{d-2}z-G)$. We see easily that the strict transform $C'=\Phi(C)$ is defined by the  equation:
\[
y^{2(d-2)}z^2=\Delta.
\]
 Write $\Delta=\prod_{i=1}^k (x-\lambda_iy)^{q_i}$,
where the $\lambda_i$'s are distinct.
Renumber $q_i$'s so that $q_i$'s are odd for $i=1,\ldots,l$ and $q_i$'s are even for $i=l+1,\ldots,k$.
Letting $s_i=[q_i/2]$ for $i=1,\ldots,k$,
we put $S=\prod_{i=1}^k (x-\lambda_iy)^{s_i}$ and $s=\sum_{i=1}^k s_i$.
Note that $2d-2=\sum_{i=1}^kq_i=2s+l$. We next define a Cremona transformation:
\[
\Psi(x,y,z)=(xS,yS, y^{s}z),
\]
We find that $\Psi^{-1}(x,y,z)=(xy^s,y^{s+1}, Sz)$. We see that $\Gamma'=\Psi(C')$ is defined by the  equation:
\[
y^{2(d-2)}z^2=\prod_{i=1}^{l}(x-\lambda_iy)
\]
We see that $l=2g+2$ and $g=d-s-2$.
Take a projective transformation: $\iota:(x,y,z)\to (x,z,y)$.
Finally, the image $\Gamma=\iota(\Gamma')$ has the affine equation:
\[
y^2=\prod_{i=1}^{2g+2}(x-\lambda_i).
\]

\bigskip
We now prove the latter half of Theorem~\ref{theorem 2}.
We start with the curve $\Gamma$ and a collection of systems of multiplicity sequences:
\[
M=\left[m,\binom{1}{1}_{b_{1}},\ldots,\binom{1}{1}_{b_{n}}, (2_{b_{n+1}}),\ldots, (2_{b_{n+n'}})\right] , 
\]
where the $m$ is the system of the multiplicity sequences of
the singular point with multiplicity $d-2$.
Let $r(M)$, $N(M)$ denote the number of the branches and the number of
the different tangent lines of $m$.
We have to construct an irreducible plane curve of type $(d,d-2)$ with
Data$(C)=M$. In \cite{SS},
we considered the case in which $g=0$.
We here assume that $g\geq 1$. We follow the arguments in \cite{SS}. 

First we deal with the cuspidal case given in Corollary of Theorem~\ref{theorem 1}. See also Proposition~\ref{equations}.

\medskip
\noindent
Case (a):
$M=[(k),(2_{b_{1}}),\ldots, (2_{b_{n'}})]$,
where $k=g+\sum b_i$.  We use the induction on $n'$.
(i) $M=[(g)]$. Interchanging coordinates, we start with the curve:
\[
\Gamma_0: x^{2g}z^2=\prod_{i=1}^{2g+2}(y-\lambda_ix).
\]
After a linear change of coordinates,
we may assume that $c=\prod_{i=1}^{2g+2}(-\lambda_i)\neq 0$.
Letting $c_1=\sqrt{c}$,
we have $\Gamma_0t=(2g)O+A_1+A'_1$,
where $A_1=(1,0,c_1), A'_1=(1,0,-c_1)$. Let $\Gamma_1$ be
the strict transform of $\Gamma_0$ via $\varphi_{c_1}$.
Using Lemma~1, (a) and Lemma 2, (e)* in \cite{SS},
we see that $\Gamma_1t=(2g-1)O+A_2$. Write $A_2=(1,0,c_2)$.
Let $\Gamma_2$ be the strict transform of $\Gamma_1$ via $\varphi_{c_2}$.
In this way, we successively choose $c_1,\ldots,c_g$.
It turns out that Data$(\Gamma_g)=[(g)]$.
(ii) Suppose we have constructed $C_0$ with
Data$(C_0)=[(k_0),(2_{b_{1}}),\ldots, (2_{b_{n'-1}})]$,
where $k_0=g+\sum_{i=1}^{n'-1} b_i$.
After a suitable change of coordinates,
we may assume $C_0l=k_0O+2B_{1}$ and $C_0t=(k_0+1)O+A_{1}$.
Note that the double covering ${\tilde C}\to \mathbf{P}^1$ defined
through the projection from $O$ to a line, must have $2g+2$ branch points.
Since $n'-1<2g+2$, we see that a line passing through $O$ is tangent to
$C_0$ at a smooth point $ B_{1}$.
Write $A_{1}=\left( 1,0,c_{1}\right) $.
Let $C_1$ be the strict transform of $C_0$ via $\varphi_{c_{1}}$.
We have  $C_1l=(k_0+1)O+2B$ and $C_1t=(k_0+2)O+A_2$.
Write $A_{2}=\left( 1,0,c_{2}\right) $.
Let $C_2$ be the strict transform of $C_1$ via $\varphi _{c_{2}}$.
We have again $C_2t=(k_0+3)O+A_3$.
Repeating in this way, we successively choose $c_1,\ldots,c_{b_{n'}}$
and define $C_1,\ldots,C_{b_{n'}}$.
Then, the curve $C=C_{b_{n'}}$ has the desired property.

\medskip
\noindent
Case (b):
$M=[(2k+1,2_{k}),(2_{b_{1}}),\ldots, (2_{b_{n'}})]$,
where $k+1=g+\sum b_i$.
As in Case~(a), we can similarly prove this case.
For the first step: $M=[(2g-1,2_{g-1})]$,
it suffices to arrange coordinates so that
$\Gamma_0t=gO+2A_1$ with $A_1=(1,0,0)$.
Put $c_1=0$ and choose $c_2,\ldots, c_g$ arbitrarily.
Then we obtain Data$(\Gamma_{g})=M$ (Cf. Lemma~1, (b) and Lemma~2,~(e)* in \cite{SS}).

\medskip
\noindent
Case (c):
$M=[(2k,2_{k+j}),(2_{b_{1}}),\ldots, (2_{b_{n'}})]$, where $k=g+j+\sum b_i$.
We also use the induction on $n'$ as in Case~(a).
For  the first step: $M=[(2(g+j),2_{g+2j})]$,
we start with a curve $C_0$ with Data$(C_0)=[(g+j), (2_j)]$ constructed
in Case~(a).
We again arrange coordinates so that $C_0t=(g+j)O+2R$,
where $\underline{m}_R(C_0)=(2_j)$ and $R=(1,0,a)$.
Choose $c_1\neq a$ and $c_2,\ldots, c_{g+j}$ arbitrarily.
Then we have Data$(C_{g+j})=M$ (Cf. Lemma~1, (a)*, (c) in \cite{SS}).   

\medskip
Starting with the cuspidal case,
we can prove the general case in a similar manner to that in \cite{SS}.
We have three subcases:
I.~$N(M)=r(M)=1$, II.~$N(M)=1, r(M)=2$, III.~$N(M)\geq 2$.
Here, we only give a proof for
$M=[(k), \binom{1}{1}_{b_1},\ldots,\binom{1}{1}_{b_n}, (2_{b_{n+1}}),\ldots,(2_{b_{n+n'}})]$,
where $k=g+\sum_{j=1}^{n+n'} b_j$, which is one of the remaining cases in I.
We use the induction on $n$.
(i) We constructed a cuspidal curve $C$ with
Data$(C)=[(k), (2_{b_{n+1}}),\ldots,(2_{b_{n+n'}})]$.
(ii) Suppose we have already constructed $C_0$ with 
\[
\mathrm{Data}(C_0)=[(k_0), \binom{1}{1}_{b_1},\ldots,\binom{1}{1}_{b_{n-1}}, (2_{b_{n+1}}),\ldots,(2_{b_{n+n'}})],
\]
where $k_0=g+\sum_{j=1}^{n-1}b_j+\sum_{j=n+1}^{n+n'}b_j$.
By arranging coordinates,
we have $C_0l=k_0O+B_1+B'_1$ and $C_0t=(k_0+1)O+A_1$.
Letting $A_1=(1,0,c_1)$, the strict transform $C_1$ of $C_0$ via $\varphi_{c_1}$ has the property $C_0t=(k_0+2)O+A_2$. 
Write $A_2=(1,0,c_2)$. 
We successively choose $c_2,\ldots,c_{b_n}$ in this way. 
Then the strict transform $C$ of $C_0$ via $\varphi_{c_{b_n}}\circ\cdots\circ \varphi_{c_1}$ has the desired property (Cf. Lemma~1,~(d) and Lemma~2,~(tn) in \cite{SS}).
In particular, $C$ contains a tacnode  $\binom{1}{1}_{b_n}$ at $B=(0,1,0)$.

\begin{remark}
Note that in Coolidge \cite{Co} (Book II, Chap.V),
the cases in which $M=[(g)]$ and $=[\binom{g}{g}\binom{1}{1}_g]$ were discussed.
\end{remark}

\section{Defining equations} We now describe the defining equations for those curves listed in Corollary. 
In \cite{FZ2,SS,ST}, the defining equations were computed step by step by using quadratic Cremona transformations. 
But, for some cases, we encountered a difficulty to evaluate points in some special positions. 
We here employ the method used by Degtyarev in \cite{D}.

\begin{lem}\label{square} Consider two polynomials
\[
g(t)=\sum_{i=0}^dc_it^i, \quad \delta(t)=\sum_{i=0}^{2d}d_it^i\in \mathrm{C}[t].
\]
Suppose $\delta(0)=d_0\neq 0$. For $k\leq d$, we have  $t^k\,|\,(g^2-\delta)$ if and only if 
\begin{enumerate}\item[{\rm (1)}] $c_0=\pm \sqrt{d_0}$,
\item[{\rm (2)}] $c_j=(d_j-\sum_{i=1}^{j-1}c_ic_{j-i})/(2c_0)$ for $j=1,\ldots,k-1$.
\end{enumerate}
\end{lem}
\begin{proof} Write $g(t)^2=\sum_{j=0}b_jt^j$. We see that $b_j=\sum_{i=0}^jc_ic_{j-i}$ for $j\leq d$.
\end{proof}

\begin{proposition}\label{equations}
The defining equations of irreducible plane curves of type $(d,d-2)$
with genus $g$ having only cusps are the following
(up to projective equivalence, the $\lambda_i$'s are distinct). 
\begin{enumerate} \item[{\rm (a)}] ${\displaystyle y^kz^2+2Gz+\Bigl\{G^2-\Delta\Bigr\}/y^k=0}$,
where \[
\Delta(x,y)=\prod_{i=1}^{n'}(x-\lambda_iy)^{2b_i+1}\prod_{i=n'+1}^{2g+2}(x-\lambda_iy).
\]
Letting $G(x,y)=\sum_{h=0}^{k+1}c_hx^{k+1-h}y^{h}$, the coefficients $c_0,\ldots,c_{k-1}$ are determined by the condition $y^k\,|\,(G(1,y)^2-\Delta(1,y))$ (See Lemma~\ref{square}).
\item[{\rm (b)}] ${\displaystyle (y^{k}z+\sum_{h=0}^{k+1}c_hx^{k+1-h}y^h)^2y-\prod_{i=1}^{n'}(x-\lambda_iy)^{2b_i+1}\prod_{i=n'+1}^{2g+1}(x-\lambda_iy)=0}$.
\item[{\rm (c)}] ${\displaystyle (y^{k}z+\sum_{h=0}^{k+1}c_hx^{k+1-h}y^h)^2-y^{2j+1}\prod_{i=1}^{n'}(x-\lambda_iy)^{2b_i+1}\prod_{i=n'+1}^{2g+1}(x-\lambda_iy)=0}$, 

\noindent
where $c_0\neq 0$.
\end{enumerate}
\end{proposition}
\begin{proof}
Class (a).
In this case,
in view of the argument in the proof of Theorem~\ref{theorem 1},~(ii),
we have $T(C)=\{(k,0),(0,2b_1+1),\ldots,(0,2b_{n'}+1), (0,1),\ldots,(0,1)\}$. Thus,
we can write $F=y^k$ and $\Delta$ as above.
We must have $y^k|(G^2-\Delta)$.
In view of Lemma~\ref{square}, the coefficients $c_0,\ldots,c_{k-1}$ are uniquely determined.
In particular, $c_0=\pm 1$.
So by Remark~\ref{criterion}, the defining equation is irreducible.

Class (b): We have $T(C)=\{(2k+1,2k+1),(0,2b_1+1),\ldots,(0,2b_{n'+1})$, $(0,1),\ldots,(0,1)\}$. We can arrange coordinates as 
\[
F=y^{2k+1}, \quad \Delta=y^{2k+1}\prod_{i=1}^{n'}(x-\lambda_iy)^{2b_i+1}\prod_{i=n'+1}^{2g+1}(x-\lambda_iy).
\]
We infer that $G=y^{k+1}G_0$ for some $G_0$. 

Class (c): We have $T(C)=\{(2k,2k+2j+1),(0,2b_1+1),\ldots,(0,2b_{n'+1})$, $(0,1),\ldots,(0,1)\}$. We can arrange coordinates as 
\[
F=y^{2k}, \quad \Delta=y^{2k+2j+1}\prod_{i=1}^{n'}(x-\lambda_iy)^{2b_i+1}\prod_{i=n'+1}^{2g+1}(x-\lambda_iy).
\] 
It follows that $G=y^kG_0$ for some $G_0$.
If we write $G_0=\sum_{h=0}^{k+1}c_hx^{k+1-h}y^h$, then we must have $c_0\neq 0$,
for otherwise the defining equation becomes reducible (See Remark~\ref{criterion}). 
\end{proof}
\begin{example}
We give the defining equation of a cuspidal septic curve $C$
with Data$(C)=[(5),(2),(2),(2),(2)]$ which are birational to the elliptic curve $y^2=(x^2-1)(x^2-\lambda^2)$, \,\,($\lambda\neq \pm 1,0$). 
\begin{multline*}
y^5z^2+\bigl\{2x^4-3(\lambda^2+1)x^2y^2+\frac{3}{4}(\lambda^4+6\lambda^2+1)y^4\bigr\}x^2z\\
-\frac{1}{8}(\lambda^2+1)(\lambda^4-10\lambda^2+1)x^6y
+\frac{3}{64}\bigl\{3\lambda^8-28\lambda^6-78\lambda^4-28\lambda^2+3\bigr\}x^4y^3\\
+3\lambda^4(\lambda^2+1)x^2y^5-\lambda^6y^7=0
\end{multline*}
\end{example}

\begin{proposition}
The defining equations of irreducible plane curves of type $(d,d-2)$
with genus $g$ having only bibranched singularities are the following
(up to projective equivalence, the $\lambda_i$'s are distinct). 
\begin{enumerate}
\item[{\rm (e)}]
${\displaystyle (y^{k}z+\sum_{h=0}^{k+1}c_hx^{k+1-h}y^h)^2-y^{2j}\prod_{i=1}^{n}(x-\lambda_iy)^{2b_i}\prod_{i=n+1}^{n+2g+2}(x-\lambda_iy)=0}$,

\noindent
where $c_0\neq 0$.

\item[{\rm (f)}] $y^{2k+r}z^2+2y^{k}G_0z+\Bigl\{G_0^2-\Delta_0\Bigr\}/y^r=0$,

\noindent
where 
\[
\Delta_0(x,y)=\prod_{i=1}^{n}(x-\lambda_iy)^{2b_i}\prod_{i=n+1}^{n+2g+2}(x-\lambda_iy)
\]
and the coefficients
$c_0,\ldots,c_{r-1}$ of $G_0(x,y)=\sum_{h=0}^{k+r+1}c_hx^{k+r+1-h}y^{h}$
are determined by the condition
$y^r\,|\,(G_0(1,y)^2-\Delta_0(1,y))$
(Cf. $Lemma~\ref{square}$) and $c_r$ is chosen so that
$y^{r+1}\,{\not|}\,\,(G_0(1,y)^2-\Delta_0(1,y))$.

\item[{\rm (aa)}]
$x^{r}y^{k}z^2+2Gz+\Bigl\{G^2-\Delta\Bigr\}/(x^ry^k)=0$,

\noindent
where 
\[
\Delta(x,y)=\prod_{i=1}^{n}(x-\lambda_iy)^{2b_i}\prod_{i=n+1}^{n+2g+2}(x-\lambda_iy) \quad (\lambda_i\neq 0 \,\,{ for \,\,all } \,\,i).
\]
Write $G(x,y)=\sum_{h=0}^{k+r+1}c_hx^{k+r+1-h}y^{h}$. 
The coefficients $c_0,\ldots,c_{k-1}$, $c_{k+2}, \ldots,c_{k+r+1}$
are determined by the conditions 
$y^k\,|\,(G(1,y)^2-\Delta(1,y))$ and $x^r\,|\,(G(x,1)^2-\Delta(x,1))$
(Cf. $Lemma~\ref{square}$).

\item[{\rm (aa+)}] $x(y^{k}z+2G_0)z+\Bigl\{xG_0^2-\Delta_0\Bigr\}/y^k=0$,

\noindent
where 
\[
\Delta_0(x,y)=\prod_{i=1}^{n}(x-\lambda_iy)^{2b_i}\prod_{i=n+1}^{n+2g+1}(x-\lambda_iy) \quad (\lambda_i\neq 0 \,\,{ for \,\,all } \,\,i).
\]
Write $G_0(x,y)=\sum_{h=0}^{k+1}c_hx^{k+1-h}y^{h}$. The coefficients $c_0,\ldots,c_{k-1}$ are determined by the condition $y^k\,|\,(G_0(1,y)^2-\Delta_0(1,y))$.

\item[{\rm (aa1)}] $xyz^2-(x-\lambda y)^{4}=0$ \,\,($\lambda\neq 0, \,\,g=0$). 

\item[{\rm (aa2)}] $xyz^2-(x-\lambda_1y)^{2}(x-\lambda_2y)^2=0$ \,\,($\lambda_1\lambda_2\neq 0, \,\,g=0$). 

\item[{\rm (aa3)}] $xyz^2-(x-\lambda_1y)^{2}(x-\lambda_2y) (x-\lambda_3y)=0$ \,\,($\lambda_1\lambda_2\lambda_3\neq 0, \,\,g=1$). 

\item[{\rm (aa4)}] $xyz^2-\prod_{i=1}^4(x-\lambda_i y)=0$ \,\,($\lambda_i\neq 0 \,\,{ for \,\,all } \,\,i, \,\,g=2$). 

\smallskip
\item[{\rm (ab)}] $x^ry^{2k+1}z^2+2y^{k+1}G_0z+\Bigl\{yG_0^2-\Delta_0\Bigr\}/x^r=0$,

\noindent
where 
\[
\Delta_0(x,y)=\prod_{i=1}^{n}(x-\lambda_iy)^{2b_i}\prod_{i=n+1}^{n+2g+1}(x-\lambda_iy) \quad (\lambda_i\neq 0 \,\,{ for \,\,all } \,\,i)
\]
and the coefficients $c_{k+2}, \ldots, c_{k+r+1}$ of $G_0(x,y)=\sum_{h=0}^{k+r+1}c_hx^{k+r+1-h}y^{h}$ are determined by the condition $x^r\,|\,(G_0(x,1)^2-\Delta_0(x,1))$. 

\item[{\rm (ab+)}] ${\displaystyle (y^{k}z+\sum_{h=0}^{k+1}c_hx^{k+1-h}y^h)^2xy-\prod_{i=1}^{n}(x-\lambda_iy)^{2b_i}\prod_{i=n+1}^{n+2g}(x-\lambda_iy)=0}$,

\noindent
where $\lambda_i\neq 0$ for all $i$.

\item[{\rm (ac)}] $x^ry^{2k}z^2+2y^{k}G_0z+\Bigl\{G_0^2-\Delta_0\Bigr\}/x^r=0$,

\noindent
where 
\[
\Delta_0(x,y)=y^{2j+1}\prod_{i=1}^{n}(x-\lambda_iy)^{2b_i}\prod_{i=n+1}^{n+2g+1}(x-\lambda_iy) \quad (\lambda_i\neq 0 \,\,{ for \,\,all } \,\,i)
\]
and the coefficients 
$c_{k+2}, \ldots, c_{k+r+1}$ of $G_0(x,y)=\sum_{h=0}^{k+r+1}c_hx^{k+r+1-h}y^{h}$ are determined by the condition
$x^r\,|\,(G_0(x,1)^2-\Delta_0(x,1))$ and $c_0\neq 0$,
which is required for the irreducibility of the defining equation
(Remark~\ref{criterion}). 

\item[{\rm (ac+)}]
${\displaystyle (y^{k}z+\sum_{h=0}^{k+1}c_hx^{k+1-h}y^h)^2x-y^{2j+1}\prod_{i=1}^{n}(x-\lambda_iy)^{2b_i}\prod_{i=n+1}^{n+2g}(x-\lambda_iy)=0}$,

\noindent
where $\lambda_i\neq 0$ for all $i$ and $c_0\neq 0$.

\item[{\rm (bb)}] 
${\displaystyle (x^{r}y^{k}z+\sum_{h=0}^{k+r+1}c_hx^{k+r+1-h}y^h)^2xy} {\displaystyle -\prod_{i=1}^{n}(x-\lambda_iy)^{2b_i}\prod_{i=n+1}^{n+2g}(x-\lambda_iy)=0}$,

\noindent
where $\lambda_i\neq 0$ for all $i$.

\item[{\rm (bc)}] ${\displaystyle (x^{r}y^{k}z+\sum_{h=0}^{k+r+1}c_hx^{k+r+1-h}y^h)^2y}\\
\quad\hfill {\displaystyle-x^{2l+1}\prod_{i=1}^{n}(x-\lambda_iy)^{2b_i}\prod_{i=n+1}^{n+2g}(x-\lambda_iy)=0}$,

\noindent
where $\lambda_i\neq 0$ for all $i$ and $c_{k+r+1}\neq 0$.

\item[{\rm (cc)}] ${\displaystyle (x^{r}y^{k}z+\sum_{h=0}^{k+r+1}c_hx^{k+r+1-h}y^h)^2}\\
\quad\hfill {\displaystyle -x^{2l+1}y^{2j+1}\prod_{i=1}^{n}(x-\lambda_iy)^{2b_i}\prod_{i=n+1}^{n+2g}(x-\lambda_iy)=0}$,

\noindent
where $\lambda_i\neq 0$ for all $i$ and $c_0c_{k+r+1}\neq 0$.

\end{enumerate}
\end{proposition}
\begin{proof}
Class (e): In this case, we have $T(C)=\{(2k,2k+2j),(0,2b_1),\ldots,$
$(0,2b_{n}),(0,1),\ldots,(0,1)\}$. We can arrange coordinates as 
\[
F=y^{2k}, \quad \Delta=y^{2k+2j}\prod_{i=1}^{n}(x-\lambda_iy)^{2b_i}\prod_{i=n+1}^{n+2g+2}(x-\lambda_iy).
\]
We infer that $G=y^kG_0$ for some $G_0$.

Class (f): We have $T(C)=\{(2k+r,2k),(0,2b_1),\ldots,(0,2b_{n}),(0,1),\ldots,$
$(0,1)\}$.
We can arrange coordinates as 
\[
F=y^{2k+r}, \quad \Delta=y^{2k}\prod_{i=1}^{n}(x-\lambda_iy)^{2b_i}\prod_{i=n+1}^{n+2g+2}(x-\lambda_iy).
\]
We infer that $G=y^kG_0$ for some $G_0$. Write $\Delta=y^{2k}\Delta_0$. Furthermore, we must have $y^r|(G_0^2-\Delta_0)$.

Class (aa): We may assume $k\geq r$. We have the case in which $T(C)= \{(k,0),(r,0),(0,2b_1),$ $\ldots,(0,2b_{n}), (0,1), \ldots, (0,1)\}$. We can then arrange coordinates as 
\[
F=x^ry^{k} \quad \Delta=\prod_{i=1}^{n}(x-\lambda_iy)^{2b_i}\prod_{i=n+1}^{n+2g+2}(x-\lambda_iy) \quad (\lambda_i\neq 0 \,\,\mathrm{ for \,\,all } \,\,i).
\]
We infer that $y^k|(G^2-\Delta)$ and $x^r|(G^2-\Delta)$. 

In case $r=1$, we also have the case in which $T(C)= \{(k,0),(1,1),(0,2b_1),$
$\ldots,(0,2b_{n})$, $(0,1), \ldots, (0,1)\}$. We obtain Class (aa+). 

If $d=4$, then we have four more classes:
\[
\begin{array}{l|l|l}
\textrm{Class}&\quad\quad\quad T(C)&g\\\hline
\textrm{(aa1)}& \{(1,1),(1,1),(0,4)\} &0\\
\textrm{(aa2)}&\{(1,1),(1,1),(0,2),(0,2)\}&0\\
\textrm{(aa3)}&\{(1,1),(1,1), (0,2), (0,1),(0,1)\}&1\\
\textrm{(aa4)}&\{(1,1), (1,1), (0,1), (0,1), (0,1), (0,1)\}&2\\ 
\end{array}
\]

For the remaining classes, we omit the details.
\end{proof} 

\bigskip
\begin{acknowledgement}
The first author would like to thank Prof. I.V.Dolgachev for comments on the book \cite{Cob}. 
\end{acknowledgement}

\bigskip
\bigskip
\begin{flushright}
\begin{tabular}{l}
Department of Mathematics \\
Faculty of Science \\
Saitama University \\
Shimo--Okubo 255 \\
Sakura--ku, Saitama 338--8570, Japan \\
\\
E--mail: \texttt{fsakai@rimath.saitama-u.ac.jp}\\
E--mail: \texttt{msaleem@yahoo.com}\\
E--mail: \texttt{ktono@rimath.saitama-u.ac.jp}%
\end{tabular}
\end{flushright}

\end{document}